\documentclass[12pt,reqno]{amsart}

\usepackage{epsfig}
\usepackage{amsmath, amsthm, amsfonts}
\usepackage{graphicx}

\numberwithin{equation}{section}

\newcommand{\R}{{\mathbb R}}

\renewcommand{\d}{\partial}

\newcommand{\f}{\varphi}

\def\average{{\mathchoice
{\kern1ex\vcenter{\hrule height.4pt width 6pt depth0pt}
\kern-9.7pt}
{\kern1ex\vcenter{\hrule height.4pt width 4.3pt depth0pt}
\kern-7pt}
{}
{}
}}

\def\aaverage{{\mathchoice
{\kern1ex\vcenter{\hrule height.4pt width 12pt depth0pt}
\kern-16pt}
{\kern1ex\vcenter{\hrule height.4pt width 12pt depth0pt}
\kern-16pt}
{}
{}
}}
\def\aave{\aaverage\iint}
\newtheorem{theo}{{\sc \bf Theorem}}[section]

\newtheorem{lem}[theo]{{\sc \bf Lemma}}

\newenvironment{defin}{\medskip\noindent{\it Definition:\/} }{\medskip}

\begin{document}

\title{Estimates in the Generalized Morrey Spaces for Linear Parabolic Systems}


\author{Matt McBride}
\address{Department of Mathematical Sciences,
Indiana University-Purdue University Indianapolis,
402 N. Blackford St., Indianapolis, IN 46202, U.S.A.}
\email{mmcbride@math.iupui.edu}

\thanks{}

\date{\today}

\begin{abstract}
The purpose of this paper is to study the parabolic system $u_t^i-D_{\alpha}(a_{ij}^{\alpha\beta} D_\beta u^j)=-\textrm{ div } f^{i}$ in the generalized Morrey Space $L_{\varphi}^{2,\lambda}$.   We would like to understand the regularity of the solutions of this system.   It will be shown that 1: if $a_{ij}^{\alpha\beta}\in C(\overline{Q_{T}})$ then $Du\in L_{\varphi}^{2,\lambda}$ , and 2: if $a_{ij}^{\alpha\beta}\in VMO(Q_{T})$ then $Du\in L_{\varphi}^{2,\lambda}$.   Moreover we will be able to obtain estimates on the gradient of the solutions to the system, which will tell us about the regularity of the solutions.
\end{abstract}

\maketitle
\section{Introduction}
In this paper we will be investigating the following linear parabolic systems of the form

\begin{equation}\label{parpde}
u_t^i-D_\alpha\left(a_{ij}^{\alpha\beta}(x,t)D_\beta u^j\right)=-\textrm{ div } f^i(x,t)\textrm{ for }i=1,\dots, N
\end{equation}

where $i,j=1,\dots, N;\ \alpha,\beta=1,\dots, n$ and the repeated indices denote summation such as

\begin{equation*}
a_{ij}\xi^i\xi^j=\sum_{i=1}^n\sum_{j=1}^na_{ij}\xi^i\xi^j.
\end{equation*}

Throughout the paper we assume an uniform ellipticity condition, namely:

\begin{equation}\label{ellpcond}
\Lambda^{-1}|\xi|^2\leq a_{ij}^{\alpha\beta}(x,t)\xi_\alpha^i\xi_\beta^j\leq\Lambda|\xi|^2
\end{equation}

where $\Lambda>0$, $\xi\in\R^{(n+1)N}$, $(x,t)\in Q_{T}$, $Q_{T}=\Omega\times[0,T]$, and $\Omega\subset\R^n$.   The main purpose of this paper is to demonstrate that one can obtain the gradient estimates in generalized Morrey spaces $L_{\varphi}^{2,\lambda}$ for weak solutions of (\ref{parpde}).   In the next section we discuss some definitions and preliminaries.

\section{Preliminaries and Weak Solutions}
In this section we discuss some needed theorems and lemmas for the main results of this paper.   More over we state some notation that will be used throughout this paper.   We denote the $n$-dimensional ball centered at $x_0$ with radius $R$ as 

\begin{equation*}
B_R(x_0) = \{x\in\R^n \ : \ \|x-x_0\| <R\}.
\end{equation*}

We will also let $z$ represent a $n+1$ dimensional coordinate, i.e. $z\in \R^n\times (0,T]$, where $z=(x,t)$, $x\in\R^n$ and $t\in(0,T]$.   Similarly $z_0 = (x_0,t_0)$.   We will denote the parabolic cylinder in $\R^{n+1}$ with vertex at $z_0$ by 

\begin{equation*}
Q_R(z_0) = B_R(x_0) \times (t_0 - R^2,t_0].
\end{equation*}

The boundary of the parabolic cylinder consists of the lateral walls, the lower boundary, and the lower corners, however we will use $\d_pQ_R$ to denote the parabolic boundary of the parabolic cylinder.   Next we define what the Morrey space is for the parabolic setting.

\begin{defin}\label{morreysp}
The parabolic Morrey space is defined to be the following

\begin{equation}
L_\f^{p,\lambda}(Q_T)= \left\{ f\in L^p(Q_T)\ :\ \underset{z_0\in Q_T,0\leq\rho\leq d}{sup }\frac{1}{\f(\rho)}\left(\rho^{-\lambda}\iint_{Q_T\cap Q_\rho(z_0)}\left|f\right|^pdz\right)^{\frac{1}{p}}<\infty\right\}
\end{equation}

with $1\leq p<\infty$, $0\leq\lambda\leq n+2$, and $\f$ is a continuous function on $\left[0,d\right]$, $\f>0$ on $(0,d]$, and $d$ is the diameter of $Q_T=\Omega\times(0,T]$ with $\Omega\subset\R^n$.
\end{defin}

We now come to the first lemma which states that the parabolic Morrey space is a Banach space.

\begin{lem}\label{bansp}
The space $L_\f^{p,\lambda}(Q_T)$ is a Banach space under the following norm

\begin{equation*}
\left\|f\right\|_{L_\f^{p,\lambda}} = \underset{z_0\in Q_T,\ 0\leq\rho\leq d}{sup }\frac{1}{\varphi(\rho)}\left(\rho^{-\lambda}\iint_{Q_T\cap Q_\rho(z_0)}\left|f\right|^p\right)^{\frac{1}{p}}
\end{equation*}
\end{lem}

\begin{proof}
First it must be shown that $\left\|\ \cdot \ \right\|_{L_\f^{p,\lambda}}$ is indeed a norm.   Then it must be shown that the space is complete.   Showing the conditions $\left\|f\right\|_{L_\f^{p,\lambda}}$ is positive definite and the homogeneity condition are trivial.   However the triangle inequality is not obvious.   Using Minkowski's inequality and the face that $\f >0$ on $(0,d]$ one gets the following,

\begin{equation*}
\begin{aligned}
&\hspace{-.8cm}\left\|f+g\right\|_{L_\f^{p,\lambda}}= \\
&=\textrm{ sup }\frac{1}{\f(\rho)}\left(\rho^{-\lambda}\iint_{Q_T\cap Q_\rho(z_0)}\left|f+g\right|^pdz\right)^{\frac{1}{p}} \\
&\leq\textrm{ sup } \frac{1}{\f(\rho)}\left(\left(\rho^{-\lambda}\iint_{Q_T\cap Q_\rho(z_0)}\left|f\right|^pdz\right)^{\frac{1}{p}}+\left(\rho^{-\lambda}\iint_{Q_T\cap Q_\rho(z_0)}\left|g\right|^pdz\right)^{\frac{1}{p}}\right) \\
&\leq\textrm{ sup }\frac{1}{\f(\rho)}\left(\rho^{-\lambda}\iint_{Q_T\cap Q_\rho(z_0)}\left|f\right|^pdz\right)^{\frac{1}{p}} + \textrm{ sup }\frac{1}{\f(\rho)}\left(\rho^{-\lambda}\iint_{Q_T\cap Q_\rho(z_0)}\left|g\right|^pdz\right)^{\frac{1}{p}} \\
&= \left\|f\right\|_{L_\f^{p,\lambda}} + \left\|g\right\|_{L_\f^{p,\lambda}}.
\end{aligned}
\end{equation*}

This string of inequalities shows the triangle inequality is satisfied and thus $\left\|\ \cdot \ \right\|_{L_\f^{p,\lambda}}$ defines a norm.   All that is left to show is that the space is complete under the norm.   To do this one must show that every Cauchy sequence from $L_\f^{p,\lambda}(Q_T)$ converges to an element in $L_\f^{p,\lambda}(Q_T)$.   Let $\{f_k\}_{k=1}^{\infty}$ be a Cauchy sequence in $L_\f^{p,\lambda}$.   Tschebyshev's inequality implies that

\begin{equation*}
m\left\{ z\in Q_T\ :\ |f_k(z)-f_m(z)|>\varepsilon\right\}\leq \varepsilon^{-p}\iint_{Q_T\cap Q_\rho(z_0)}\left|f_k-f_m\right|^pdz
\end{equation*}

where $m$ is standard Lebesgue measure.   Therefore, there exists a subsequence $\{f_{k_j}\}$
and a $f$ such that $\{f_{k_j}\}$ converges to $f$ a.e. in $Q_{T}$.   Then for every $\varepsilon>0$ there exists $K$ such that $\|f_{k_j}-f_k\|_{L_\f^{p,\lambda}}<\varepsilon$ if $k_j,k>K$.    Letting $k_j\rightarrow\infty$, Fatou's lemma implies that
 $\|f-f_k\|_{L_\f^{p,\lambda}}<\varepsilon$ for $k>K$.    Thus $f\in L_\f^{p,\lambda}$ by $\|f\|_{L_\f^{p,\lambda}}\leq\|f-f_k\|_{L_\f^{p,\lambda}}+\|f_k\|_{L_\f^{p,\lambda}}<\infty$ and $\|f-f_k\|_{L_\f^{p,\lambda}}\rightarrow0$ as $k\rightarrow\infty$
.   Therefore $L_\f^{p,\lambda}$ is complete and hence it is a Banach space.   This completes the proof.
\end{proof}

For the rest this paper we will set $p=2$ in the Morrey space $L_\f^{p,\lambda}$.   We state this next for convenience.

\begin{equation*}
L_\f^{2,\lambda}(Q_T)= \left\{ f\in L^2(Q_T)\ :\ \underset{z_0\in Q_T,0\leq\rho\leq d}{sup }\frac{1}{\f(\rho)}\left(\rho^{-\lambda}\iint_{Q_T\cap Q_\rho(z_0)}\left|f\right|^2 dz\right)^{\frac{1}{2}}<\infty\right\}
\end{equation*}

The next definition was coined in \cite{Hu} and the subsidiary lemma was also proven in the same paper.   The following definition defines when a function is said to be ``almost'' increasing.   We say ``almost'' since the natural thought with almost would be that the function is increasing everywhere except on a set of measure zero, however, here is not the case.

\begin{defin}\label{almst_inc}
A function $h:\left[0,d\right]\rightarrow[0,\infty)$ is said to be almost increasing if there exists $K_h\geq1$ such that $h(s)\leq K_hh(t)$ for $0\leq s\leq t\leq d$.
\end{defin}

Now we state the lemma that was proven in \cite{Hu}.

\begin{lem}\label{almst_inc_lem}
Let $H$ be a non-negative almost increasing function in $[0,R_0]$ and $F$ a positive function on $(0,R_0]$.    Suppose that

\begin{enumerate}
  \item There exists $A,B,\varepsilon,\beta>0$ such that $H(\rho)\leq\left(A\left(\rho/R\right)^{\beta}+\varepsilon\right)H(R)+BF(R)$ for 
  \newline $0\leq\rho\leq R\leq R_{0}$

\item There exists $\gamma\in\left(0,\beta\right)$ such that $\frac{\rho^{\gamma}}{F(\rho)}$ is almost increasing in $(0,R_{0}]$
\end{enumerate}

Then there exists $\varepsilon_0=\varepsilon_0(A,\beta,\gamma)$ and $C=C\left(A,\beta,\gamma,K_H,K\right)$ such that if $\varepsilon<\varepsilon_0$ then $H(\rho)\leq C\frac{F(\rho)}{F(R)}H(R)+CBF(\rho)$.
\end{lem}

Next we define the bounded mean oscillation and vanishing mean oscillation spaces in the parabolic setting.    This is already well understood in say a $n$-dimensional ball or on some bounded domain $\Omega$ in $\R^n$.   We define the bounded mean oscillation, $BMO(Q_T)$, in the following way.

\begin{defin}
Let $\psi\in C[0,d]$ and $\psi>0$ on $[0,d]$; so $\psi$ is a positive continuous function on the interval $[0,d]$.   The $BMO(Q_T)$ is defined by

\begin{equation*}
\begin{aligned}
&BMO_{\psi}(Q_{T})= \\
&=\left\{ f\in L^2(Q_T)\ :\ \underset{z_0\in Q,0\leq\rho\leq d}{sup}\frac{1}{\psi(\rho)}\left(\aave_{Q_T\cap Q_\rho(z_0)}\left|f(z)-f_{Q_T\cap Q_\rho(z_0)}(z_0)\right|^2 dz\right)^{\frac{1}{2}}<\infty\right\}
\end{aligned}
\end{equation*}

where $f_A={\displaystyle \aave_A f(z)dz = \frac{1}{m(A)}\iint_A f(z)dz}$ and $A\subset\R^{n+1}$.
\end{defin}

Next we define the vanishing mean oscillation space for the parabolic setup in a similar way.

\begin{defin}
If $\psi=1$, where $\psi$ is the continuous function defined in the bounded mean oscillation space definition, then the $VMO(Q_T)$ is defined by

\begin{equation*}
VMO(Q_{T}) =\left\{ f\in BMO(Q_{T})\ :\ [f]_{BMO(Q_T:\sigma)}\rightarrow0\textrm{ as }\sigma\rightarrow0\right\}
\end{equation*}

where

\begin{equation*}
[f]_{BMO(Q_T:\sigma)} = \underset{z_0\in Q,0\leq\rho\leq\sigma}{sup}\left(\aave_{Q_T\cap Q_\rho(z_0)}\left|f(z)-f_{Q_T\cap Q_\rho(z_0)}(z_0)\right|^2 dz\right)^{\frac{1}{2}}
\end{equation*}
\end{defin}

For the rest of this section discuss the weak solutions to variation of the system of parabolic parabolic partial differential equations that this paper is concerned with.   Let $a_{ij}^{\alpha\beta}$ be constant and consider the following system in $Q_T$ :

\begin{equation}\label{homogen_pde}
u_t^i-D_\alpha\left(a_{ij}^{\alpha\beta}D_\beta u^j\right)=0
\end{equation}

For $Q_R(z_0)\subset Q_T$, let $u^i\xi^2(x)\eta(t)$ be a test function with $\xi\in C_0^{\infty}(B_R(x_0))$, the space of smooth functions vanishing at infinity,  $0\leq\xi\leq1$, and $\left|D\xi\right|\leq\frac{C}{R-\rho}$ with $B_\rho(x_0)\subset B_R(x_0)\subset\Omega$ and $\eta(t)$ defined in the following way

\begin{equation*}
\eta(t) = \left\{
\begin{array}{cc}
\frac{t-(t_0-R^2)}{R^2-\rho^2} & t\in (t_0-R^2,t_0-\rho^2) \\
1 & t\in[t_0-\rho^2,t_0)
\end{array}\right. .
\end{equation*}

Multiplying (\ref{homogen_pde}) by the test function, using integration by parts and noticing that the boundary term is zero by the definition of $\eta$ and $\xi$ one gets

\begin{equation*}
\begin{aligned}
0&=\iint_{B_R(x_0)\times(t_0-R^2,t]}\left(u_t^i-D_\alpha\left(a_{ij}^{\alpha\beta}D_\beta u^j\right)\right)u^i\xi^2\eta\ dxdt\\
&=\iint_{B_R(x_0)\times(t_0-R^2,t]}u_t^i u^i\xi^2\eta+a_{ij}^{\alpha\beta}D_\beta u^j D_\alpha\left(u^i\xi^2\eta\right)\ dxdt\\
&=\iint_{B_R(x_0)\times(t_0-R^2,t]}\left(\frac{1}{2}\left|u\right|^2\right)_t\xi^2\eta\ dxdt \\
&\qquad +\iint_{B_R(x_0)\times(t_0-R^2,t]}a_{ij}^{\alpha\beta}D_\beta u^j\left(\xi^2 D_\alpha u^i + 2\xi u^i D_\alpha\xi\right)\eta\ dxdt \\
&=\iint_{B_R(x_0)\times(t_0-R^2,t]}\left(\frac{1}{2}\left|u\right|^2 \eta\right)_t \xi^2 -\frac{1}{2}\left|u\right|^2 \xi^2 \eta_t\ dxdt \\
&\qquad + \iint_{B_R(x_0)\times(t_0-R^2,t]}a_{ij}^{\alpha\beta}D_\beta u^j\left(\xi^2 u^i +2 \xi u^i D_\alpha \xi\right)\eta\ dxdt .
\end{aligned}
\end{equation*}

Then by the uniform ellipticity condition, (\ref{ellpcond}), and the Cauchy-Schwartz inequality one has

\begin{equation*}
\begin{aligned}
\int_{B_R(x_0)}&\frac{1}{2}\left|u(x,t)\right|^2 \xi^2(x)\ dx + C\int_{t_0-R^2}^t\int_{B_R(x_0)}\xi^2(x)\left|Du\right|^2\ dxdt \\
&\leq\frac{1}{2}\int_{t_0-R^2}^t\int_{B_R(x_0)}\left|u\right|^2\xi^2\eta_t\ dxdt +C\int_{t_0-R^2}^t\left|D\xi\right|^2\left|u\right|^2\eta\ dt \\
&\leq C\int_{t_0-R^2}^t\int_{B_R(x_0)}\left|u\right|^2\left(\left|D\xi\right|^2\eta +\frac{1}{2}\xi^2\eta_t\right)\ dxdt .
\end{aligned}
\end{equation*}

Then since $|D\xi|\leq \frac{C}{R-\rho}$ and, by a simple computation, $\eta_t \leq \frac{C}{R^2-\rho^2}$ one gets

\begin{equation*}
\begin{aligned}
\int_{B_R(x_0)}&\frac{1}{2}\left|u(x,t)\right|^2\xi^2\ dx + \int_{t_0-R^2}^t \int_{B_R(x_0)}\left|Du\right|^2\xi^2\eta\ dxdt \\
&\leq C\int_{t_0-R^2}^t\int_{B_R(x_0)}\left|u\right|^2\left(\frac{1}{(R-\rho)^2}+\frac{1}{R^2-\rho^2}\right)\ dxdt .
\end{aligned}
\end{equation*}

This last inequality in turn implies the following inequality

\begin{equation}\label{energy_estimate}
\underset{t_0-\rho^2 \leq t\leq t_0}{\textrm{ sup }}\int_{B_\rho(x_0)}|u|^2 +\iint_{Q_\rho(z_0)}|Du|^2\leq\frac{C}{(R-\rho)^2}\iint_{Q_R(z_0)}|u|^2
\end{equation}

The inequality, (\ref{energy_estimate}), is called the energy estimate for the system of partial differential equations stated in (\ref{homogen_pde}).   We derived this energy estimate as it will have applications to the proofs of the main results it also lets us define a Sobolev space counterpart for parabolic equations.   Consider the following definition of a space

\begin{equation*}
V_2(Q_T)=\left\{ u\ :\ u\in L^\infty(0,T;L^2(Q_T)),\ Du\in L^2(Q_T)\right\} .
\end{equation*}

$V_2(Q_T)$ is said to be the Solobev space counterpart for parabolic equations.   We call it this since our definition of it looks very similar to the definition of the Solobev space with $p=q=2$.   Using these energy estimates and the Sobolev embedding theorem, one can get the Morrey estimate for the system of partial differential equations stated in (\ref{homogen_pde}) with constant coefficients.    The Morrey estimate for a system of homogeneous parabolic partial differential equations with constant coefficients is

\begin{equation}
\iint_{Q_\rho(z_0)}|Du|^2\leq C\left(\frac{\rho}{R}\right)^{n+2}\iint_{Q_R(z_0)}|Du|^2
\end{equation}

for $Q_R(z_0)\subset Q_T$ and $0\leq \rho\leq R$.   We end this section with the formal statement of this and a proof.

\begin{lem}\label{const_est}
Let $u\in V_2(Q_T)$ be a solution to the system of partial differential equations defined in (\ref{homogen_pde}) in $Q_T = \Omega\times (0,T]$.   Then for $Q_R(z_0)\subset Q_T$ and $0\leq \rho\leq R$ the following inequality holds

\begin{equation*}
\iint_{Q_\rho(z_0)}|Du|^2\leq C\left(\frac{\rho}{R}\right)^{n+2}\iint_{Q_R(z_0)}|Du|^2
\end{equation*}
\end{lem}

\begin{proof}
Recall the system of partial differential equations stated in (\ref{homogen_pde}):

\begin{equation}
u_t^i-D_\alpha\left(a_{ij}^{\alpha\beta}D_\beta u^j\right)=0 .
\end{equation}

Since the coefficients, $a_{ij}^{\alpha\beta}$, are constant, differentiating the above equation with respect to $x$ shows that $D_\alpha u$ is still a solution to the system of differential equations.   By \cite{Sc} one has

\begin{equation*}
\aave_{Q_\rho(z_0)}|u|^2 \leq C\left(\frac{\rho}{R}\right)^2\aave_{Q_R(z_0)}|u|^2
\end{equation*}

where $u$ is a solution to the above system of equations.   Using this inequality and the fact that $D_\alpha u$ is still a solution one arrives at

\begin{equation*}
\iint_{Q_\rho(z_0)}|Du|^2\leq C\left(\frac{\rho}{R}\right)^{n+2}\iint_{Q_R(z_0)}|Du|^2
\end{equation*}

and the result follows.   Thus this completes the proof.
\end{proof}

We now close this section and turn the main results of this paper.

\section{Main Results}
In the previous section the Morrey estimate we are interested in was shown when the coefficients, $a_{ij}^{\alpha\beta}$, were constant and mainly done by \cite{Sc}.   In this section we will extend the result to the system of partial differential equations defined in (\ref{parpde}).   We first establish the Morrey estimate for the case $a_{ij}^{\alpha\beta}\in C(\overline{Q_T})$, the space of continuous functions on the closure of $Q_T$, and second the case $a_{ij}^{\alpha\beta}\in L^\infty(Q_T)\cap VMO(Q_T)$, the space of all bounded functions with vanishing mean oscillation.

\begin{theo}\label{cont_estimate}
Let $u\in V_2(Q_T)$ be a weak solution, in $Q_T$, to the following system of partial differential equations

\begin{equation*}
u_t^i-D_\alpha\left(a_{ij}^{\alpha\beta}(z)D_\beta u^j\right)=-div\, f^i
\end{equation*}

for $i=1,\dots, N$.   Let $a_{ij}^{\alpha\beta}\in C(\overline{Q_T})$ and suppose they satisfy the uniform ellipticity condition with $f^i\in L_\f^{2,\lambda}(Q_T)$.   Suppose there exists $\lambda$ and $\gamma$ such that $\lambda<\gamma<n+2$ and that the function $\frac{r^{\gamma-\lambda}}{\f(r)}$ is almost increasing, then $Du\in L_\f^{2,\lambda}(Q')$ for any $Q'\subset\subset Q_T$, for $Q_R(z_0)\subset Q_T$ and $\rho\leq R$.   Moreover the following inequality holds

\begin{equation*}
\iint_{Q_\rho(z_0)}|Du|^2dz\leq C\frac{\rho^\lambda\f^2(\rho)}{R^\lambda\f^2(R)}\iint_{Q_R(z_0)}|Du|^2dz+C\f^2(\rho)\rho^\lambda\|f\|_{L_\f^{2,\lambda}}^2
\end{equation*}

\end{theo}

\begin{proof}
Let $w$ satisfy the following system

\begin{equation}\label{fixed_pt_eq}
\left\{
\begin{array}{cc}
w_t^i-D_\alpha\left(a_{ij}^{\alpha\beta}(z_0)D_\beta w^j\right)=0 & \textrm{ in } Q_R(z_0) \\
w=u & \textrm{ on } \d_pQ_R(z_0)
\end{array}\right.
\end{equation}

where $z_0$ is a fixed point.   Then $v=u-w$ will satisfy this system

\begin{equation}\label{u_minus_w}
\left\{
\begin{array}{cc}
v_t^i-D_\alpha\left(a_{ij}^{\alpha\beta}(z_0)D_\beta v^j\right) = D_\alpha\left(\left(a_{ij}^{\alpha\beta}(z)-a_{ij}^{\alpha\beta}(z_0)\right)D_\beta u^j\right)-\textrm{ div }f^i & \textrm{ in } Q_R(z_0) \\
v=0 & \textrm{ on } \d_pQ_R(z_0)
\end{array}\right. .
\end{equation}

Clearly by Lemma (\ref{const_est}) one obtains the following inequality

\begin{equation*}
\begin{aligned}
\iint_{Q_\rho(z_0)} |Du|^2 &\leq 2\iint_{Q_\rho(z_0)}\left(|Dw|^2+|Dv|^2\right) \\
&\leq C\left(\frac{\rho}{R}\right)^{n+2} \iint_{Q_R(z_0)} |Dw|^2 + \iint_{Q_\rho(z_0)} |Dv|^2 \\
&\leq C\left(\frac{\rho}{R}\right)^{n+2} \iint_{Q_R(z_0)} |Dw|^2 + C\iint_{Q_R(z_0)} |Dv|^2 .
\end{aligned}
\end{equation*}

Multiplying equation (\ref{u_minus_w}) by $v$, integrating and performing integration by parts, one obtains the following

\begin{equation*}
\iint_{Q_R(z_0)} v_t^iv^i + \iint_{Q_R(z_0)} a_{ij}^{\alpha\beta}(z_0)D_\beta v^jD_\alpha v^i \leq \iint_{Q_R(z_0)}\left|a_{ij}^{\alpha\beta}(z)-a_{ij}^{\alpha\beta}(z_0)\right||Du||Dv| + |f||Dv| .
\end{equation*}

Since $a_{ij}^{\alpha\beta}\in C(\overline{Q_T})$, for small enough $R$ one has $\left|a_{ij}^{\alpha\beta}(z)-a_{ij}^{\alpha\beta}(z_0)\right|< \varepsilon$ for some $\varepsilon>0$.   Therefore one gets

\begin{equation*}
\begin{aligned}
\iint_{Q_R(z_0)}& v_t^iv^i + \iint_{Q_R(z_0)} a_{ij}^{\alpha\beta}(z_0)D_\beta v^jD_\alpha v^i \\
&\leq \varepsilon \iint_{Q_R(z_0)}|Du||Dv| + \iint_{Q_R(z_0)} |f||Dv| \\
&\leq \varepsilon \iint_{Q_R(z_0)}\left(|Du|^2 + |Dv|^2\right) + \iint_{Q_R(z_0)} |f|^2
\end{aligned}
\end{equation*}

using the Schwartz inequality.   This last inequality yields

\begin{equation*}
\iint_{Q_R(z_0)}|Dv|^2 \leq \varepsilon \iint_{Q_R(z_0)} |Du|^2 + C\iint_{Q_R(z_0)}|f|^2 .
\end{equation*}

Therefore one obtains the following

\begin{equation*}
\begin{aligned}
\iint_{Q_\rho(z_0)}&|Du|^2 \\
&\leq C\left(\frac{\rho}{R}\right)^{n+2} \iint_{Q_R(z_0)} |Du|^2 + \varepsilon \iint_{Q_R(z_0)}|Du|^2 +C\iint_{Q_R(z_0)} |f|^2 \\
&\leq \left(C\left(\frac{\rho}{R}\right)^{n+2} + \varepsilon\right)\iint_{Q_R(z_0)} |Du|^2 + C\f^2(R)R^\lambda \|f\|_{L_\f^{2,\lambda}}^2 .
\end{aligned}
\end{equation*}

Then the desired result follows immediately from lemma (\ref{almst_inc_lem}).   Thus the theorem has been proved.
\end{proof}

It has just been shown that the Morrey estimate is valid for continuous functions on the closure of the parabolic domain, which can be thought of as a finite cylinder.   Before the desired estimate can be proven, we will need two additional lemmas.   The first lemma, dubbed the ``reverse'' H\"{o}lder inequality can be found in \cite{SG}.   We will use this in proving the second lemma that we need before the final estimate is shown.

\begin{lem}\label{rev_hol_ineq}
Let $u\in V_2(Q_T)$ be a weak solution to the following system

\begin{equation*}
u_t^i - D_\alpha\left(a_{ij}^{\alpha\beta}(z)D_\beta u^j\right) =0
\end{equation*}

in $Q_T$ with $i=1,\dots, N$.   Assume that the $a_{ij}^{\alpha\beta}$ satisfy the uniform ellipticity condition, then there exists some $s>2$ such that $Du\in L_{loc}^s (Q_T)$ and for every $Q_R\subset Q_{4R}\subset Q_T$ the following inequality holds

\begin{equation*}
\left(\aave_{Q_R} |Du|^s dz\right)^{\frac{1}{s}} \leq C\left(\aave_{Q_{4R}} |Du|^2 dz\right)^{\frac{1}{2}}
\end{equation*}
\end{lem}

We leave without proof as this was done in \cite{SG}.   We now move to the second lemma needed for the last Morrey estimate that we desire.

\begin{lem}\label{int_est}
Let $u\in V_2(Q_T)$ be a weak solution to the following system

\begin{equation*}
u_t^i - D_\alpha\left(a_{ij}^{\alpha\beta}(z)D_\beta u^j\right) =0
\end{equation*}

in $Q_T$ with $i=1,\dots, N$.   Assume that the $a_{ij}^{\alpha\beta}\in L^\infty(Q_T)\cap VMO(Q_T)$ and that they satisfy the uniform ellipticity condition.   Then for any $0<\mu < n+2$ there exists $R_0$ and $C$ depending only on $n+2$, $N$, $\mu$, $\Lambda$, and $\left[a_{ij}^{\alpha\beta}\right]_{BMO(Q_T;\sigma)}$ such that for $\rho\leq R\leq \frac{1}{2}\textrm{min}(R_0,\textrm{dist}(z_0,\d_pQ_T))$ the following inequality holds

\begin{equation*}
\iint_{Q_\rho(z_0)} |Du|^2 dz \leq C\left(\frac{\rho}{R}\right)^\mu \iint_{Q_R(z_0)} |Du|^2 dz
\end{equation*}
\end{lem}

\begin{proof}
First define the following

\begin{equation*}
\left(a_{ij}^{\alpha\beta}\right)_{z_0R} := \aave_{Q_R(z_0)} a_{ij}^{\alpha\beta}(x,t)\ dxdt .
\end{equation*}

As in theorem (\ref{cont_estimate}), let $w$ and $v=u-w$ satisty the following systems of partial differential equations respectively

\begin{equation}
\left\{
\begin{array}{cc}
w_t^i-D_\alpha\left(a_{ij}^{\alpha\beta}(z_0)D_\beta w^j\right)=0 & \textrm{ in } Q_R(z_0) \\
w=u & \textrm{ on } \d_pQ_R(z_0)
\end{array}\right.
\end{equation}

and

\begin{equation}\label{sys_pde}
\left\{
\begin{array}{cc}
v_t^i-D_\alpha\left(\left(a_{ij}^{\alpha\beta}\right)_{z_0R} D_\beta v^j\right) = D_\alpha\left(\left(a_{ij}^{\alpha\beta}(z)-\left(a_{ij}^{\alpha\beta}\right)_{z_0R}\right)D_\beta u^j\right) & \textrm{ in } Q_R(z_0) \\
v=0 & \textrm{ on } \d_pQ_R(z_0)
\end{array}\right. .
\end{equation}

Similar to the proof of theorem (\ref{cont_estimate}) one can obtain the following estimate

\begin{equation*}
\iint_{Q_\rho(z_0)}|Du|^2\leq C\left(\frac{\rho}{R}\right)^{n+2}\iint_{Q_R(z_0)}|Du|^2 + C\iint_{Q_R(z_0)} |Dv|^2 .
\end{equation*}

Multiplying $v$ to equation (\ref{sys_pde}), integrating and performing an integration by parts, one gets the following inequality

\begin{equation*}
\iint_{Q_R(z_0)}|Dv|^2\leq C\iint_{Q_R(z_0)} \left|a_{ij}^{\alpha\beta}(z)-\left(a_{ij}^{\alpha\beta}\right)_{z_0R}\right|^2|Du|^2 .
\end{equation*}

Using H\"{o}lder's inequality, lemma (\ref{rev_hol_ineq}) and with the fact that the $a_{ij}^{\alpha\beta}\in VMO$ one obtains the following inequality

\begin{equation*}
\begin{aligned}
\iint_{Q_R(z_0)}|Dv|^2 &\leq C\left(\iint_{Q_R(z_0)} \left|a_{ij}^{\alpha\beta}(z)-\left(a_{ij}^{\alpha\beta}\right)_{z_0R}\right|^{2p}\right)^{\frac{1}{p}}\left(\iint_{Q_R(z_0)}|Du|^{2q}\right)^{\frac{1}{q}} \\
&=C\cdot m\left(Q_R(z_0)\right)\left(\aave_{Q_R(z_0)} \left|a_{ij}^{\alpha\beta}(z)-\left(a_{ij}^{\alpha\beta}\right)_{z_0R}\right|^{\frac{2s}{s-2}}\right)^{\frac{s-2}{s}}\left(\aave_{Q_R(z_0)}|Du|^{s}\right)^{\frac{2}{s}} \\
&\leq C\cdot m\left(Q_R(z_0)\right)\varepsilon\left(\aave_{Q_R(z_0)}|Du|^s\right)^{\frac{2}{s}} \\
&\leq C\varepsilon \iint_{Q_{4R}(z_0)} |Du|^2
\end{aligned}
\end{equation*}

Therefore one gets the following inequality 

\begin{equation*}
\begin{aligned}
\iint_{Q_\rho(z_0)}|Du|^2&\leq \left(C\left(\frac{\rho}{R}\right)^{n+2} +\varepsilon\right) \iint_{Q_{4R}(z_0)} |Du|^2 \\
&\leq \left(C\left(\frac{\rho}{R}\right)^{n+2} +\varepsilon\right) \iint_{Q_R(z_0)} |Du|^2 .
\end{aligned}
\end{equation*}

Then by using lemma (\ref{almst_inc_lem}) one can achieve the desired result.   Thus this completes the proof.
\end{proof}

We are now in the position to state and prove the final theorem in this paper.   We have all the tools necessary to get the final Morrey estimate which extends the result in the system of linear elliptic partial differential equations case.

\begin{theo}
Let $u\in V_2(Q_T)$ be a weak solution to the following system of parabolic partial differential equations

\begin{equation*}
u_t^i - D_\alpha\left(a_{ij}^{\alpha\beta}(z)D_\beta u^j\right) = - \textrm{div }f^i
\end{equation*}

in $Q_T$ with $i=1,\dots, N$ and let the $a_{ij}^{\alpha\beta}$ satisfy the uniform ellipticity condition.   Suppose there exists $\lambda$ and $\gamma$ such that $\lambda < \gamma< n+2$ and that the function $\frac{r^{\gamma- \lambda}}{\f^2(r)}$ is almost increasing.   If the $a_{ij}^{\alpha\beta}\in L^\infty(Q_T)\cap VMO(Q_T)$ and $f^i \in L_\f^{2, \lambda}(Q_T)$, then $Du\in L_\f^{2, \lambda}(Q')$ for any $Q'\subset\subset Q_T$ and for $Q_R\subset Q_T$ and $\rho \leq R$.   Moreover the following interior integral estimate holds

\begin{equation*}
\iint_{Q_\rho}|Du|^2 dz \leq C \frac{\rho^\lambda \f^2(\rho)}{R^\lambda \f^2(R)}\iint_{Q_R}|Du|^2 dz + C\f^2(\rho)\rho^\lambda\|f\|_{L_\f^{2, \lambda}}^2
\end{equation*}

\end{theo}

\begin{proof}
Again as in theorem (\ref{cont_estimate}), let $w$ and $v=u-w$ satisfy the following systems of partial differential equations respectively

\begin{equation}\label{sys_1}
\left\{
\begin{array}{cc}
w_t^i-D_\alpha\left(a_{ij}^{\alpha\beta}(z)D_\beta w^j\right)=0 & \textrm{ in } Q_R(z_0) \\
w=u & \textrm{ on } \d_pQ_R(z_0)
\end{array}\right.
\end{equation}

and

\begin{equation}\label{sys_2}
\left\{
\begin{array}{cc}
v_t^i-D_\alpha\left(a_{ij}^{\alpha\beta}(z)D_\beta v^j\right)=-\textrm{div }f^i & \textrm{ in } Q_R(z_0) \\
v=0 & \textrm{ on } \d_pQ_R(z_0)
\end{array}\right. .
\end{equation}

Applying lemma (\ref{int_est}) to $w$ yields the following inequality

\begin{equation}\label{eq_1}
\iint_{Q_\rho(z_0)}|Du|^2 \leq C\left(\frac{\rho}{R}\right)^\mu \iint_{Q_R(z_0)} |Du|^2 + C\iint_{Q_R(z_0)} |Dv|^2.
\end{equation}

Then multiplying $v$ to equation (\ref{sys_2}), integrating and performing an integration by parts, just as in theorem (\ref{cont_estimate}), one has the following 

\begin{equation*}
\iint_{Q_R(z_0)} |Dv|^2 \leq C \iint_{Q_R(z_0)} |f||Dv| .
\end{equation*}

Using the Cauchy-Schwartz inequality one this last inequality one gets

\begin{equation}\label{eq_2}
\iint_{Q_R(z_0)} |Dv|^2 \leq C \iint_{Q_R(z_0)} |f|^2 .
\end{equation}

Since one has $f^i \in L_\f^{2, \lambda}(Q_T)$, combining equations (\ref{eq_1}) and (\ref{eq_2}) one obtains the following inequality

\begin{equation*}
\iint_{Q_\rho(z_0)}|Du|^2 dz \leq C \left(\frac{\rho}{R}\right)^\mu\iint_{Q_R(z_0)}|Du|^2 dz + C\f^2(\rho)\rho^\lambda\|f\|_{L_\f^{2, \lambda}}^2 .
\end{equation*}

Therefore the result will follow by applying lemma (\ref{almst_inc_lem}).   Thus this completes the proof.
\end{proof}

\end{document}